\documentclass[letterpaper,12pt]{article}
\usepackage{amsthm}
\usepackage{amsmath}
\usepackage{placeins}
\usepackage{appendix}
\usepackage{amssymb}
\usepackage{color}
\usepackage{xspace}
\usepackage{url}
\usepackage{booktabs}
\usepackage{thmtools}
\usepackage{cleveref}
\usepackage{mathpazo}       
\usepackage[scaled]{helvet} 
\usepackage{courier}        
\normalfont
\declaretheorem[name=Theorem]{thm}
\declaretheorem[name=Proposition]{prop}
\declaretheorem[name=Lemma]{lemma}
\newcommand{\Z}{\mathbb{Z}}
\newcommand{\Q}{\mathbb{Q}}
\newcommand{\legendre}[2]{\ensuremath{\left(\frac{#1}{#2}\right)}}

\newcommand{\Ep}{\ensuremath{\mathcal{E}}}
\newcommand{\Bp}{\ensuremath{\mathcal{B}}}

\newcommand{\divides}{\mid}
\newcommand{\notdivides}{\nmid}

\crefname{equation}{}{}
\Crefname{equation}{}{}

\title{Infinitely Many Carmichael Numbers for a Modified Miller-Rabin Prime
  Test}
\author{
\newlength{\bachwidth}
\newlength{\rexwidth}
\settowidth{\bachwidth}{bach@cs.wisc.edu}
\settowidth{\rexwidth}{rex@cs.wisc.edu}
\begin{minipage}{\bachwidth}
\centering Eric Bach\footnotemark[1]\\
bach@cs.wisc.edu
\end{minipage}
\qquad
\begin{minipage}{\rexwidth}
\centering Rex Fernando\footnotemark[1] \\
rex@cs.wisc.edu
\end{minipage}
\bigskip\\
University of Wisconsin - Madison \\
1210 W Dayton St. \\
Madison, WI 53706}
\date{November 2015}

\begin{document}
\renewcommand*{\thefootnote}{\fnsymbol{footnote}}
\maketitle
\footnotetext[1]{Research supported by NSF: CCF-1420750}
\renewcommand*{\thefootnote}{\arabic{footnote}}

\begin{abstract}
  We define a variant of the Miller-Rabin primality test, which is in between
  Miller-Rabin and Fermat in terms of strength. We show that this test has
  infinitely many ``Carmichael'' numbers. We show that the test can also be
  thought of as a variant of the Solovay-Strassen test. We explore the growth
  of the test's ``Carmichael'' numbers, giving some empirical results and a
  discussion of one particularly strong pattern which appears in the results.
\end{abstract}

\section{Introduction}

Primality testing is an important ingredient in many cryptographic protocols.
There are many primality testing algorithms; two important examples are Solovay
and Strassen's test~\cite{ss}, and Rabin's modification~\cite{rabin} of a test by
Miller~\cite{miller}, commonly called the Miller-Rabin test.  Solovay-Strassen has
historical significance because it was proposed as the test to be used as part
of the RSA cryptosystem in~\cite{rsa}, arguably one of the most important
applications of primality testing. Miller-Rabin is the more widely used of the
two tests, because it achieves a small error probability more
efficiently than Solovay-Strassen. A notable example of Miller-Rabin's usage is
in the popular OpenSSL secure communication library~\cite{openssl}. 

We explore the relationship between these two tests and the much older Fermat
test. Both tests can be thought of as building upon the Fermat test; indeed,
all three algorithms have a very similar structure, but the Fermat test has a
fatal weakness which the two more modern tests fix: as Alford, Granville and
Pomerance proved in~\cite{agp}, there is an infinite set of composite numbers
which in effect fool the Fermat test, causing it to report that they are prime.
These numbers are called \emph{Carmichael numbers}, after the discoverer of the first
example of such a number~\cite{carmichael}.

We now give descriptions of the three algorithms. All
three take an odd number \(n \in Z\) to be tested for primality, and start by
choosing a random \(a \in \Z\), where \(2 \leq a \leq n-1\). The Fermat test,
the simplest of the three, checks whether \(a^{n-1} \equiv
1\pmod{n}\). If so, it returns ``Probably Prime'', and if not it returns
``Composite''. The Solovay-Strassen computes the Jacobi symbol
\(\legendre{a}{n}\), and returns ``Composite'' if \(\legendre{a}{n} = 0\) or
\(a^{(n-1)/2} \not\equiv \legendre{a}{n} \pmod{n}\). Otherwise it returns
``Probably prime''.  Let \(n-1 = 2^r \cdot d\)
with \(d\) odd. The Miller-Rabin test considers the sequence 
\begin{align}\label{intro:seq}
  a^d, a^{2 d}, ..., a^{2^{r-1} d}, a^{2^r d};
\end{align}
if \(1\) does not appear in the
sequence, or if it appears directly after \(-1\), then the test returns
``Composite''; otherwise it returns ``Probably Prime''.

We can think of both of newer algorithms as being more specific versions of the
Fermat test. Both essentially perform the Miller-Rabin test, but each also
performs
some extra work, so as to avoid the fatal weakness of the Fermat test. An
interesting question, then, is "Why is this extra computation necessary?" The
infinitude of Carmichael numbers can be thought of as an answer to this
question, in a sense. We explore this question further below.

In particular, we study the following variant of the Miller-Rabin test. Fix
some constant \(z\). Instead of checking the whole sequence \Cref{intro:seq},
only check the last \(z+1\) numbers. In the case where \(z=1\), the test can be
thought of as the following variant of Solovay-Strassen: after generating
\(a\), check whether \(a^{(n-1)/2} \equiv \pm 1 \pmod{n}\). These two
variants are both more specific than Fermat, but less specific than the
respective tests they are based on. The main result of this paper shows that
when Miller-Rabin and Solovay-Strassen are weakened in this way, both tests
behave more like the Fermat test than before, namely there are infinitely
``Carmichael'' numbers for both tests. Thus, just as the infinitude of
Carmichael numbers explains why the Fermat test is not good enough, our result
explains why all the added work in Miller-Rabin is necessary.

Let \(C_z(x)\) denote the number of ``Carmichael'' numbers less than \(x\) for our variant of
Miller-Rabin with parameter \(z\). The contributions of this paper are:
\begin{itemize}
  \item A lower bound on \(C_z(x)\), of the same strength as Alford, Granville
and Pomerance's lower bound on the number of Carmichael numbers and based on
their work.
\item An empirical comparison of \(C_z(x)\) to \(C(x)\), the number of
  Carmichael numbers less than \(x\).
\item Two heuristic arguments suggesting that the ratio \(C_z(x)/C(x)\) decays
      exponentially.
\end{itemize}

The organization of this paper is as follows.  \Cref{preliminaries} contains
relevent preliminaries. \Cref{mainresult} contains the main result, and
\Cref{otherthings} contains the upper bound discussion and empirical results.

%

\section{Overview of \protect\cite{agp}'s Original Argument} \label{preliminaries}

We describe the argument used in \cite{agp} to prove there are infinitely many
carmichael numbers.

By Korselt's criterion~\cite{korselt} a positive composite integer \(n>1\) is a
Carmichael number iff it is odd and squarefree and for all primes \(p\)
dividing \(n\), \(n \equiv 1 \pmod{p-1}\). The approach of \cite{agp} uses this
criterion and exploits following theorem, proved by multiple independent
parties (see the discussion in~\cite{agp}). 

\begin{thm}[2 in \protect\cite{agp}]\label{algthm}
  If \(G\) is a finite abelian group in which the maximal order of an element
  is \(m\), then in any sequence of at least \(m(1 + \log(|G|/m))\) (not
  necessarily distinct) elements of \(G\), there is a nonempty subsequence
  whose product is the identity.
\end{thm}

Given this theorem, assume we have an odd integer \(L\), and we can find many
primes \(p\) where \(p-1\) divides \(L\). If there are enough such primes, some
of them must multiply to equal the identity in \((\Z/L\Z)^*\). The product of
those primes is then a Carmichael number, by Korselt's criterion. This strategy
was suggested by Erd\"os~\cite{erdos} as a way to prove there are infinitely
many Carmichael numbers, although he did not know \Cref{algthm} and simply
guessed that there might be a way to exhibit many products that produce the
identity. \cite{agp} successfully implemented a modified version of this strategy. We
state the main theorem in \cite{agp} before continuing. Here \(\Ep\) is a set of
positive number-theoretic constants related to choosing \(L\), and \(\Bp\) is
another set of constants related to finding primes in arithmetic progressions
(see~\cite{agp}). Let \(C(X)\) be the number of Carmichael numbers less than \(X\).

\begin{thm}[1 in \protect\cite{agp}]
  For each \(E \in \Ep\) and \(B \in \Bp\) there is a number \(x_0 = x_0(E,B)\)
  such that \(C(x) \geq x^{EB}\) for all \(x \geq x_0\).
\end{thm}

At the time the best results for \(\Ep\) and \(\Bp\) allowed the exponent to be
\(2/7\). The exponent has since been improved slightly;
see~\cite{harman,harman2}.

To achieve this result, \cite{agp} show there is an \(L\) (parameterized by
\(X\)) where \(n((\Z/L\Z)^*)\) is relatively small compared to \(L\). Ideally,
they would have then shown that there are many primes \(p\) where \(p-1 | L\). But
the best they could show was that there is some \(k < L^c\) for
some \(c < 1\) where there are many primes \(p\) that satisfy \(p-1 | kL\).
This is from a theorem by Prachar~\cite{prachar}. This is not as convenient,
because now the group in question is \(G = (\Z/kL\Z)^*\), whose largest order
\(m\) is not necessarily small. \cite{agp} gets around this by modifying
Prachar's theorem to guarantee that \((k,L) = 1\) and for each \(p, p \equiv 1
  \pmod{k}\). These primes are in the subgroup of \((\Z/kL\Z)^*\) of residue
classes that are \(1\mod{k}\), which is isomorphic to \((\Z/L\Z)^*\), thus
fixing the problem.  They used a simple counting argument based on
\Cref{algthm} to show the existence of enough products of primes chosen from
the set of \(p\) to satisfy the lower bound claimed.

\section{Depth \(z\)} \label{mainresult}

We restate the Miller-Rabin variant described in the introduction. Given an odd
positive integer \(n\) to test for primality, choose \(a\) at random from
\(\Z_n^*\). Let \(n-1 = 2^r \cdot d\) with \(d\) odd. The original Miller-Rabin
uses the sequence 

\begin{align}\label{mainseq}
  a^d, a^{2 d}, ..., a^{2^{r-1} d}, a^{2^r d};
\end{align}

if \(1\) does not appear in the sequence, or if it appears directly after
\(-1\), then the test returns "Composite"; otherwise it returns "Probably
Prime." Our variant, which we refer to as the \emph{\(z\)-deep Miller-Rabin
  test} (with parameter \(z\)), performs the same check, but only considers the
last \(z+1\) numbers in the sequence. (If there are fewer than \(z+1\) numbers
in the sequence it looks at all of them.) Note that the \(0\)-Miller-Rabin test
is simply the Fermat test. 

We define a \emph{\(z\)-deep Carmichael number} to be a composite number \(n\) which
fools the \(z\)-deep Miller-Rabin test for all \(a \in \Z_n^*\). We have the
following claim:

\begin{prop}\label{zkorselt}
  \(n\) is a \(z\)-deep Carmichael number iff it is odd and squarefree and for
  all \(p \divides n,\) \((p-1) \divides \frac{n-1}{2^z}\).
\end{prop}

The proof is similar to the proof of Korselt's criterion and is left to the
reader. As before, \(C_z(x)\) is the number of \(z-\)deep Carmichael numbers less than \(x\).
Our goal is to prove the following theorem. 

\begin{thm}\label{mainthm}
    Choose any constant \(z \in \Z^+\). For each \(E \in \Ep, B \in \Bp\) and
    \(\epsilon > 0\), there is a number \(x_4(E,B,\epsilon)\), such that
    whenever \(x \geq x_4(E,B,\epsilon)\), we have \(C_z(x) \geq x^{EB -
        \epsilon}\).
\end{thm}

We now introduce our modification of the argument in~\cite{agp}. Carmichael
numbers are constructed in~\cite{agp} from sequences of primes which are of the
form \(p = dk + 1\) where \(d \divides L\) and for some \(k \leq L^c\), \(c <
  1\).  Let \(k = 2^\nu l\). We want to constrain each constructed Carmichael
number \(n\) to be \(\equiv 1 \pmod{2^{\nu+z}}\); if we can achieve this, then
the resulting numbers will be \(z\)-deep. Banks and Pomerance~\cite{banks}
modifiy the method in~\cite{agp} to constrain the constructed Carmichael
numbers to be \(1\) modulo some given constant number. (This is a simple
subcase of their general result.) Going beyond this, we show that \(k\) can be constrained so that
\(\nu\) is bounded above by a constant. Then we use the result in~\cite{banks}
to show there are infinitely many Carmichael numbers which are \(1
  \pmod{2^{\nu+z}}\), proving \Cref{mainthm}.

\subsection{Bounding \(\nu\)}

\cite{agp} choose \(k\) during their proof of the modified Prachar's
Theorem, which we now state. Recall that \(B\) is one of the two
number-theoretic constants which \cite{agp} relies on throughout their paper.

\begin{thm}[3.1 in~\protect\cite{agp}]
  There exists a number \(x_3(B)\) such that if \(x \geq x_3(B)\) and \(L\)
  is a squarefree integer not divisible by any prime exceeding
  \(x^{(1-B)/2}\) and for which \(\sum_{\mbox{prime }q\divides L} 1/q \leq
    (1-B)/32\), then there is a positive integer \(k \leq x^{1-B}\) with
  \((k,L) = 1\) such that 

  \[\#\{d\divides L : dk + 1 \leq x, dk + 1 \mbox{ is prime}\} \geq
    \frac{2^{-D_B-2}}{\log x}\#\{d \divides L : 1 \leq d \leq x^B\}.\]
\end{thm}

We sketch~\cite{agp}'s proof. It involves showing that for each divisor \(d
  < x^B\) of \(L\) (excluding some troublesome divisors) the number of primes
\(p \leq dx^{1-B}\) with \(p \equiv 1\mod{d}\) and \(((p-1)/d,L)=1\) is
large, and then by choosing \(k\) to be the \((p-1)/d\) that shows up the
most. The lower bound on the number of such primes \(p\) is achieved by
taking the number of primes \(p \leq dx^{1-B}\) with \(p\equiv 1\mod{d}\),
and then subtracting the number of primes less than \(dx^{1-B}\) that are 
\(1 \mod dq\) for any prime \(q\divides L\):
\[\pi(dx^{1-B};d,1) - \sum_{\mbox{prime }q\divides L} \pi(dx^{1-B};dq,1).\]
\cite{agp} use a lower bound which they derive to show \[\pi(dx^{1-B};d,1) \geq
  \frac{dx^{1-B}}{2\phi(d)\log{x}},\] and the Brun-Titchmarsh upper
bound~\cite{mv} to show \[\pi(dx^{1-B};dq,1) \leq
  \frac{8}{q(1-B)}\frac{dx^{1-B}}{\phi(d)\log{x}}.\]
It then follows that

\begin{align*}
\pi(dx^{1-B};d,1) &- \sum_{\text{prime }q\divides L} \pi(dx^{1-B};dq,1) \\
&\geq \left(\frac{1}{2} - \frac{8}{1-B}\sum_{\text{prime }q\divides L} 
  \frac{1}{q}\right)\frac{dx^{1-B}}{\phi(d)\log{x}} \geq
\frac{x^{1-B}}{4\log{x}},
\end{align*}
the last bound following from the assumption that \(\sum_{\mbox{prime
    }q\divides L} 1/q \leq (1-B)/32\). This concludes our sketch.  

\medbreak

Our goal is to get the same result with the added guarantee that the
largest power of \(2\) that divides \(k\) is small. We add the additional condition that
\(p \not\equiv 1\mod{2^{\nu_0}d}\), where \(\nu_0 \in \Z^+\) is a constant chosen
so that \(\frac{1-B}{32} > \frac{1}{2^{\nu_0}}\). So the number of such primes
\(p\) becomes

\[\pi(dx^{1-B};d,1) - \sum_{\mbox{prime }q\divides L} \pi(dx^{1-B};dq,1) -
  \pi(dx^{1-b};2^{\nu_0}d,1).\]

By the same bound as before, \(\pi(dx^{1-B};2^{\nu_0}d,1) \leq
  \frac{8}{2^{\nu_0}(1-B)}\frac{dx^{1-B}}{\phi(d)\log{x}}\), thus

\begin{align*}
\pi(dx^{1-B};d,1) &- \sum_{\mbox{prime }q\divides L} \pi(dx^{1-B};dq,1) -
  \pi(dx^{1-b};2^{\nu_0}d,1) \\
&\geq \left(\frac{1}{2} - \frac{8}{1-B}\left(\sum_{\text{prime }q\divides L} 
    \frac{1}{q} +
    \frac{1}{2^{\nu_0}}\right)\right)\frac{dx^{1-B}}{\phi(d)\log{x}}.
  \end{align*}
This requires \(\sum_{\text{prime }q\divides L}\frac{1}{q} \leq
  \frac{1-B}{32} - \frac{1}{2^{\nu_0}}\) in order to result in the same lower
bound of \(\frac{x^{1-B}}{4\log{x}}\), which is a stronger assumption than
before; but this turns out not to be a problem (explained later). The result of
all the above is our modified version of \cite{agp}'s Theorem 3.1:

\begin{thm}\label{prachar}
  Choose any \(\nu_0 \in \Z^+\) so that \(\frac{1-B}{32} >
    1/2^{\nu_0}\). Then there exists a number \(x_3(B)\) such that if
  \(x \geq x_3(B)\) and \(L\) is a squarefree integer not divisible by any
  prime exceeding \(x^{(1-B)/2}\) and for which \(\sum_{\mbox{prime }q\divides
      L} 1/q \leq (1-B)/32 - 1/2^{\nu_0}\), then there is a positive integer
  \(k \leq x^{1-B}\) with \((k,L) = 1\) and \(2^{\nu_0} \notdivides k\) such that 

  \[\#\{d\divides L : dk + 1 \leq x, dk + 1 \mbox{ is prime}\} \geq
    \frac{2^{-D_B-2}}{\log x}\#\{d \divides L : 1 \leq d \leq x^B\}.\]
\end{thm}

\subsection{The Modified~\protect\cite{agp}}

We follow the~\cite{agp} method using our new result above (\Cref{prachar}) and a new
choice for \(G\), as in~\cite{banks}.
Let \(G = (\Z/2^{\nu_0+z}L\Z)^*\). By the Chinese Remainder
Theorem, if there is a sequence whose product is the identity in
\(G\), then the product is both \(1 \mod{L}\) and \(1\mod{2^{\nu_0+z}}\).  We
use this \(G\) instead of \((\Z/L\Z)^*\). If we denote by \(n(G)\) the largest
sequence of elements of \(G\) which does not have a subsequence that multiplies
to the identity, then this choice of \(G\) does not change the upper bound on
\(n(G)\) given in \cite{agp}'s original argument. \cite{agp} parameterizes the proof of their
main theorem on some \(y\) sufficiently large, and calculates both \(L\) and the outwardly
visible parameter \(x\) based on \(y\). The upper bound on \(n(G)\)
parameterized by \(y\), given originally in equation (4.4) in~\cite{agp},
becomes \[n(G) < 2^{\nu_0+z}\lambda(L)(1+\log{2^{\nu_0+z}L})
  \leq e^{3\theta y},\] with the right hand side not changing. 

The last issue is the one mentioned above, that \(\sum_{\text{prime } q |
    L}\frac{1}{q}\) must be less than \(\frac{1-B}{32} - \frac{1}{2^{\nu_0}}\)
instead of just \(\frac{1-B}{32}\) The reason why this is not a problem is that
AGP shows \(\sum_{\text{prime } q | L}\frac{1}{q} \leq 2\frac{\log\log
    y}{\theta \log y}\), which is actually asymptotically less than any
  constant.

The changes we have made have only affected the minimum choice of \(x\) for which
the proof will work; the other logic of the proof is not affected. So for
large enough \(x\) we get the same fraction of sequences whose products are
\(1\) in \(G\).  Since any product of such a sequence \(\prod(S)\) is
\(1\pmod{L}\) it follows that the product is \(1\pmod{kL}\) and thus a
Carmichael number. Any prime number \(p \in S\) is of the form \(dk+1\) with
\(d\) odd and \(k = 2^\nu l\) and \(2^{\nu+z} \divides \prod(S)-1\), so \(p-1 =
  dk | \frac{\prod(S)-1}{2^z}\).  Hence, we have a proof of \Cref{mainthm}.

\section{Upper Bound and Empirical Results}\label{otherthings}

\begin{table}[t]
  \centering
  \begin{tabular}{r r r r r r r r r}
    \toprule
    \# Prime factors: & 3 & 4 & 5 & 6 & 7 & 8 &  All \\
    \midrule
 \(z=0\) & 1166 & 2390 & 3807 & 2233 & 388 & 16 & 10000 \\
    1    & 498  & 1244 & 1834 & 1090 & 204 & 8  & 4878  \\
    2    & 239  & 586  & 916  & 553  & 99  & 6  & 2399  \\
    3    & 110  & 297  & 462  & 298  & 48  & 3  & 1218  \\
    4    & 52   & 139  & 232  & 142  & 23  & 1  & 589   \\
    5    & 26   & 76   & 108  & 75   & 13  & 1  & 299   \\
    6    & 12   & 39   & 49   & 40   & 6   & 0  & 146   \\
    7    & 10   & 20   & 21   & 21   & 0   &    & 72    \\
    8    & 2    & 12   & 10   & 11   &     &    & 35    \\
    9    & 0    & 8    & 2    & 5    &     &    & 15    \\
    10   &      & 4    & 1    & 3    &     &    & 8     \\
    11   &      & 3    & 1    & 2    &     &    & 6     \\
    12   &      & 2    & 0    & 2    &     &    & 4     \\
    13   &      & 2    &      & 1    &     &    & 3     \\
    14   &      & 0    &      & 1    &     &    & 1     \\
    \bottomrule
  \end{tabular}
  \caption{The number of depth-\(z\) Carmichael numbers up to 1713045574801 (the
    10000th Carmichael number), filtered by number of prime factors.}
  \label{depth_table}
\end{table}

From the OEIS' list of the first \(10,000\) Carmichael numbers~\cite{oeis_carmichael},
we tallied the numbers which are \(z\)-deep Carmichaels for \(z=1\) to
\(14\), the maximum depth observed. We also separated the counts by the number
of prime factors up to \(8\), the maximum number observed. The
results are in \Cref{depth_table}.

Observe that \(C_{z}(x)\)
is about \(1/2^z\) of \(C(x)\). It would be interesting to prove this
rigorously. We now discuss two points which make progress
in this direction. First is an observation about the proof of the
latest upper bound for \(C(x)\), given in~\cite{psw80} and improved
in~\cite{pomerance81}. We observe that the dominant term in the proof of the
upper bound follows the pattern in the table. Second is a heuristic
idea to support the pattern of halving the number of Carmichaels with each
increase in depth. Although they are far from rigorous, they do allow for some
qualitative predictions.

\subsection{The Dominant Term in the Carmichaels Upper Bound}

\newcommand{\csum}{\sum_{x^{1-2\delta} < k \leq x^{1-\delta}} x/kf(k)}

Let \(\ln_k x\) denote the \(k\)-fold iteration of \(\ln\). In
1980~\cite{psw80} proved the following:

\begin{thm}[6 in~\protect\cite{psw80}]
  For each \(\epsilon > 0\), there is an \(x_0(\epsilon)\) such that for all
  \(x \geq x_0(\epsilon)\), we have \(C(x) \leq x \exp{-(1-\epsilon)\ln x \cdot
      \ln_3 x / \ln_2 x}\)
\end{thm}

See~\cite{psw80}, p. 1014. We outline their proof here. Let
\(\delta > 0\). Divide the Carmichael numbers \(n \leq x\) into three classes:

\begin{align*}
  N_1 &= \mbox{\# Carmichaels } n \leq x^{1-\delta} \\
  N_2 &= \mbox{\# Carmichaels } x^{1-\delta} < n \leq x \mbox{ where } n \mbox{
    has a prime factor } p \geq x^\delta \\
  N_3 &= \mbox{\# Carmichaels } x^{1-\delta} < n \leq x \mbox{ where all prime
    factors of } n \mbox{ are below } x^\delta
\end{align*}

We get that \(N_1 \leq x^{1-\delta}\) trivially, and \(N_2 <
  2x^{1-\delta}\) (see~\cite{psw80} for details).

\cite{psw80} show 
\begin{align}
N_3 \leq x^{1-\delta} + \csum, \label{n3bound}
\end{align}
where \(f(k)\) is the
least common multiple of \(p-1\) for all \(p \divides k\). The sum in
\Cref{n3bound} is the
dominating term in the sum \(N_1 + N_2 + N_3 = C(x)\). We show how to
strengthen this term for \(C_z(x)\).

\begin{prop}\label{upperbound}
  The number of \(z\)-deep Carmichael numbers \(n \leq x\) divisible by some
  integer \(k\) is at most \(1 + x/2^zkf(k)\).
\end{prop}
\begin{proof}
  Any such \(n\) is \(0 \pmod{k}\) and \(1
    \pmod{2^zf(k)}\). The latter congruence is because \(n \equiv 1
    \pmod{f(k)}\) and \(n \equiv 1 \pmod{2^{z+y}}\), where y is the largest
  number such that \(2^y \divides p-1\) for some \(p \divides n\) prime. So
  \(2^zf(k)\) and \(k\) are coprime, and the result follows by the Chinese
  Remainder Theorem.
\end{proof}

With this lemma, and the observation that any \(n\) in the third class has a
\(k \divides n\) where \(x^{1-2\delta} < k \leq x^{1-\delta}\), we have that
\[N_3 \leq x^{1-\delta} + \frac{1}{2^z}\csum.\]

It is possible to also derive similar bounds for \(N_1\) and \(N_2\), in order
to show that \[C_z(x) < \frac{1}{2^z}x\exp{(-(1-\epsilon)\ln x
    \cdot \ln_3 x / \ln_2 x)}.\] This does not improve the bound asymptotically,
though, since if \(\epsilon_1 < \epsilon_2\) then \[x\exp{(-(1-\epsilon_1)\ln x
    \cdot \ln_3 x / \ln_2 x)}  < \frac{1}{2^z} x\exp{(-(1-\epsilon_2)\ln x \cdot
    \ln_3 x / \ln_2 x)}\] asymptotically for any \(z\). Nevertheless, we still
find this interesting. Pomerance~\cite{pomerance81}
sharpens the estimate for the sum in \Cref{n3bound} to get a slightly better upper bound for
\(C(x)\), and conjectures that this upper bound is tight. Assuming this is the
case, the sum in \Cref{n3bound} is the most important term in determining the growth of
\(C(x)\). Additionally, \Cref{upperbound} fits almost perfectly with the data in
\Cref{depth_table}.

\subsection{The Local Korselt Criterion}

\def \Znstar {{\Z}/(n)^*}
\def \Ztwo {{\Z}_2}

Let $n$ be a composite number, and recall $\lambda(n)$ is
the maximum order of any element of $\Znstar$.  If $p$ is
prime, we say that $n$ is {\sl $p$-Korselt} if 
$\nu_p(\lambda(n)) \le \nu_p(n-1)$.  For example, 33 is 2-Korselt 
but 15 is not.  This is a local version of the Korselt criterion.
Indeed, $n$ is a Carmichael number iff it is $p$-Korselt for every $p$,
and satisfies a global property (squarefree with at
least 3 prime factors).

Let $n$ be a Carmichael number, say $n = p_1 p_2 ... p_r$.  Then
$$
   \nu_2 (n-1) - \max_i \left\{ \nu_2 (p_i - 1) \right\} \ge 0.
$$
We say $n$ has exact depth $z$ if this difference is $z$.  
By \Cref{zkorselt}, then, ``depth $z$'' is the same as 
``exact depth $\ge z$.''

To study this situation, we shall model
$p_1,p_2,\ldots,p_r$ by i.i.d. random elements of
$\Ztwo^*$ (invertible 2-adic integers).  In binary notation, such
a number is written
$$
   \cdots b_4 b_3 b_2 b_1 1 .
$$
Here $b_i \in \{0,1\}$ for $i \ge 1$.  Our model amounts to
imagining that these bits are chosen by independent flips of 
a fair coin.

Let $\nu_i = \nu_2(p_i-1)$. (We are abusing notation here.)
Note first that if all $\nu_i$ are equal, then $p_1 p_2 \cdots p_r$
is 2-Korselt.  We now distinguish three cases.

First, let $r$ be odd, with the exponents $\nu_i$ equal.  
Then we have
\begin{align*}
           p_1 & = 1 + u_1 2^{\nu} \\
           p_2 & = 1 + u_2 2^{\nu} \\
               & \vdots            \\
           p_r & = 1 + u_r 2^{\nu} 
\end{align*}
with each $u_i$ odd.  Their product is
$$
1 + \left( \sum_{i=1}^r u_i \right) 2^\nu + 
    \hbox{ [ terms divisible by $2^{\nu+1}$ ] }
\equiv 1 + 2^\nu \pmod {2^{\nu+1}},
$$
so the exact depth is 0.

Second, let $r$ be even, with the exponents $\nu_i$ equal.  Then,
\begin{align*}
           p_1 & = 1 + 2^\nu + u_1 2^{\nu+1} \\
           p_2 & = 1 + 2^\nu + u_2 2^{\nu+1} \\
               & \vdots            \cr
           p_r & = 1 + 2^\nu + u_r 2^{\nu+1} \\
\end{align*}
with $u_i$ arbitrary.  Since $r$ is even and $\nu \ge 1$,
$$
p_1 \cdots p_r = 
1 + \left( r/2 + \sum_{i=1}^r u_i \right) 2^{\nu + 1} + 
    \hbox{ [ terms divisible by $2^{\nu+2}$ ] }
\equiv 1 \pmod {2^{\nu+1}},
$$
so the depth is at least 1.  As a consequence of this equation,
$$
p_1 \ldots p_r - 1 =
p_1 \ldots p_{r-1}(1 + 2^\nu) - 1 + p_1 \ldots p_{r-1} u_r 2^{\nu+1} \equiv 0 \pmod {2^{\nu+z}}
$$
iff
$$
\frac{p_1 \ldots p_{r-1}(1 + 2^\nu) - 1} {2^{\nu+1}} + p_1 \ldots p_{r-1} u_r 
      \equiv 0 \pmod {2^{z-1}}.
$$
This determines $u_r$ mod $2^{z-1}$, making the probability 
of depth $\ge z$ equal to $1/2^{z-1}$, for $z \ge 1$.

Finally, let the exponents $\nu_i$ be unequal.  Let
$\nu := \max_i \{ \nu_2 (p_i-1) \} = \nu_2(p_1-1)$.  
Then we may write
\begin{align*}
           p_1 & = 1 + \qquad \qquad u_1 2^\nu \\
           p_2 & = 1 + 2x_2 + u_2 2^\nu \\
               & \vdots            \\
           p_r & = 1 + 2x_r + u_r 2^\nu \\
\end{align*}
with $0 \le x_2,\ldots,x_r < 2^{\nu-1}$, and (without loss of
generality) $x_r \ne 0$.  Whether 2-Korselt holds depends
entirely on the $x_i$'s.  If it does, we have
$$
p_1 \cdots p_r - 1 =
p_1 \cdots p_{r-1}(1 + 2x_r) - 1 + p_1 \cdots p_{r-1} u_r 2^{\nu}
\equiv 0 \pmod {2^{\nu+z}}
$$
iff
$$
\frac{ p_1 \cdots p_{r-1}(1 + 2x_r) - 1 } {2^\nu} + p_1 \cdots p_{r-1} u_r 
\equiv 0 \pmod {2^z}
$$
Since the coefficient of $u_r$ is odd, this congruence has
one solution.  Therefore, for unequal exponents,
$$
\Pr[ \hbox{ depth $\ge z$ } | \hbox{ 2-Korselt } ] = 1/2^z.
$$

To summarize, we have the following result.

\begin{thm}
  Let \(p_1,\dots,p_r\) be randomly chosen odd \(2\)-adic integers,
  with \(r \geq 3\).  Let \(z \geq 1\).  Under the condition that
  \(p_1, \ldots, p_r\) is \(2\)-Korselt,

  $$
  \Pr[ \hbox{ depth $\ge z$ } ] = \begin{cases}
    1 & \hbox{if $r$ is odd and all $\nu_2(p_i - 1)$ are equal} \\
    1/2^{z-1} & \hbox{if $r$ is even and all $\nu_2(p_i - 1)$ are equal} \\
      1/2^z & \hbox{otherwise} \\
  \end{cases}
  $$
\end{thm}

\begin{table}[t]
  \centering
\begin{tabular}{r r r r r r r r r}
\toprule
$r$            &3    &4    &5    &6    &7    &8    &9    &10        \\
\midrule
\hbox{count}   &4299 &2600 &2533 &1951 &1830 &1573 &1471 &1314      \\
$N/r$          &3333 &2500 &2000 &1667 &1428 &1250 &1111 &1000      \\
\bottomrule
\end{tabular}
\caption{2-Korselt $r$-tuple counts ($N=10000$ samples).}
\label{local_table}
\end{table}

In our local model, what is the probability that
$p_1 \cdots p_r$ is 2-Korselt?
To study this, we first ran simulations, taking each
$p_i$ to be $1 + 2R_i$, with $R_i$ an 12-digit pseudorandom integer.
The Monte Carlo results, given in Table 2, suggest that
$\Pr[ \hbox{ $p_1 p_2 \cdots p_r$ is 2-Korselt } ] = \Theta(1/r)$.

Further analysis, which we give in the appendix, reveals that
$$\Pr[ \hbox{ $p_1 p_2 \cdots p_r$ is 2-Korselt } ] \in \Q$$
and that this is indeed $\Theta(1/r)$.  Our computations
match the observations. For example the observed fraction
for $r=3$ is close to the exact probability $3/7 = 0.428571...$.

Observe that the fraction of tuples $p_1,...,p_r$ for which
all $\nu_i$ are equal is
$$
   2^{-r} + 2^{-2r} + 2^{-3r} + \cdots = \frac 1 {2^r - 1}.
$$
Since the fraction of 2-Korselt $r$-tuples is $\Theta(1/r)$,
we can draw the following conclusion about the local model:
Ignoring the equal-exponent case, whose frequency diminishes
with increasing $r$, the fraction of 2-Korselt $r$-tuples with
depth $z$ (that is, exact depth $\ge z$)
decreases geometrically, with multiplier 1/2.

We conjecture, therefore, that for every \(z \geq 1\),
$$
\lim_{x \rightarrow \infty} C_z(x) / C(x) = 1/2^z.
$$

Moreover, if $C_z^{(r)}(x)$ and $C^{(r)}(x)$ denote similar counts
for Carmichaels with $r$ prime factors, there is a constant $c_r$
such that
$$
\lim_{x \rightarrow \infty} C^{(r)}_z(x) / C^{(r)}(x) = c^{(z)}_r,
$$
and $c^{(z)}_r \rightarrow 2^{-z}$ as $r$ increases.

Let us look at \Cref{depth_table} in this light.  The prediction seems accurate
for overall counts, but becomes less so when $z$ and $r$
are small.  For example, the local model predicts that 1/3 of the 
2-Korselt numbers for $r=3$ will have depth 1 (this was checked by simulation).
However, the actual fraction in our population of Carmichaels
is $498/1166 = 0.427101...$.

We do not have an explanation for this, but we can point out
two weaknesses in the local model.  First, it ignores
the odd primes.  Second, it assumes that the
$p_i$ are independent, when in fact they interact 
(e.g. $\sum_{i=1}^r \nu_i \le \log_2 n$).

What would it take to make the heuristic argument rigorous?
First, we would need to know that the prime number theorem
for arithmetic progressions still held, when the primes were
restricted to those appearing in Carmichael numbers. Second,
we would need a precise understanding of the deviation from
independence for primes appearing together in a Carmichael
number.  (That there is a deviation is clear, since the number of primes that
are $3 \pmod{4}$ has to be even.)

\bigskip

It is also of interest to consider the prime factor distribution
in \Cref{depth_table}.  By the Erd\H{o}s-Kac theorem, a random number $\le n$
has about $\log\log n  + M$ distinct prime factors, where 
$M \doteq 0.261497$ is the Mertens constant.  For \Cref{depth_table}, we have
$n = 1.71305 \times 10^{12}$, so this mean is $\lambda = 3.59973$.
However, we don't have random numbers, since every Carmichael
has at least 3 prime factors.  The conditional expectation
can be reckoned as follows.  One of the standard
models for the number of prime factors is the Poisson
distribution.  Let $Z \sim P(\lambda)$.  Under this hypothesis,
$$
        E[Z; Z \ge  3] = 3.14719, \qquad
      \Pr[Z \ge 3] = 0.697205.
$$
Dividing the first by the second gives us a prediction of
4.5140.  On the other hand, the actual average (computed from 
the top row of the table) is 4.8335.  The relative
error is about 7\%.

\bigskip



\FloatBarrier

\bibliographystyle{alpha}
\bibliography{paper}

\appendix

\def \Znstar {{\Z}/(n)^*}
\def \Ztwo {{\Z}_2}
\def \Ztwostar {{\Z}_2^*}

\section{Probability of 2-Korselt}

We now establish the probability of that an \(r\)-tuple is 2-Korselt in our model.

\begin{lemma} Let $X_1,\ldots,X_n$ be i.i.d. random variables having
a geometric distribution with parameter 1/2.  (So $X_i$ is 1 with
probability 1/2, 2 with probability 1/4, and so on.)  Let
$Z = \max_i \{ X_i \}$.  Then
$$
W(n) := \sum_{k \ge 1} \Pr[Z=k] \cdot \frac{1}{2^{k-1}}
= \sum_{j=0}^n (-1)^j {n \choose j} \frac{1}{2^{j+1} - 1}.
$$
\end{lemma}
\begin{proof} Let $P_n(k) = \Pr[ Z \le k ] = (1 - 2^{-k})^n$.
Applying partial summation,
$$
W(n) = \sum_{k \ge 1} \frac{1}{2^{k-1}}
       \left[ P_n(k) - P_n(k+1) \right]
       = \sum_{k \ge 1} \frac{1}{2^k} (1 - 2^{-k})^n.
$$
To obtain the result, expand the $n$-th powers by the binomial
theorem, interchange the order of summation, and sum the
resulting geometric series.
\end{proof}


\begin{thm}\label{korseltprob}
Let $p_1,\ldots,p_r$ be random elements of $\Ztwostar$, with $r \ge 3$.
Then $p_1 p_2 \cdots p_r$ is 2-Korselt with probability
$$
\frac{1}{2^r - 1} \left[
         1 + \sum_{2 \le s < r \atop s\ {\rm even}}
         {r \choose s}
         \sum_{j=0}^{r-s} (-1)^j {r-s \choose j} \frac 1 {2^{j+1} - 1}
         \right].
$$
\end{thm}

Before proving this, let us fix notation.
Let $p_i = 1 + u_i 2^{\nu_i}$ with $u_i$ odd, for $i=1,\ldots,r$.
(Almost surely, $p_i \ne 1$, so $\nu_i \ge 1$).
Let's call $(\nu_1,\ldots,\nu_r)$ the {\sl exponent vector}
of $p_1,\ldots,p_r$.  Let $\mu = \min_i \{ \nu_i \}$ and
$\nu = \max_i \{ \nu_i \}$.

It is an interesting fact that the 2-Korselt property constrains
the exponent vector.  In particular, unless the exponents are
equal, the minimum exponent $\mu$ must occur an even number of times.
To prove this, suppose there are $s$ copies of $\mu$, and $s<r$.
If $s$ is odd,
$$
p_1 p_2 \cdots p_r \equiv 1 + \sum_{i=1}^s u_i 2^\mu
                   \equiv 1 + 2^\mu \pmod {2^{\mu+1}} ,
$$
which cannot be 1 mod $2^{\nu}$.  This holds for Carmichael
numbers as well.  We have not seen this observation in the literature,
although it is known for $\mu=1$. (We thank Andrew Shallue for informing us
about this.)

Now to prove \Cref{korseltprob}.  We will exploit the principle of deferred
decisions~\cite{knuth} , which is a ``dynamic'' way of thinking about
conditional probability.

Imagine that we reveal bits of the $p_i$'s in parallel (taking
blocks of $r$ at a time), until the minimal exponent $\mu$ is known.
Then, the $p_i$'s look like this:
$$
\begin{matrix}
                \cdots  * * 1 0 \cdots 0 0 1. \\
                \cdots  * * 1 0 \cdots 0 0 1. \\
                \cdots  * * 1 0 \cdots 0 0 1. \\
                \cdots  * * 0 0 \cdots 0 0 1. \\
                                \vdots        \\
                \cdots  * * 1 0 \cdots 0 0 1. \\
                    \longleftarrow \ \rm time \\
\end{matrix}
$$
In this picture, the *'s stand for bits that are not yet revealed.
Note that the block of bits immediately to their right is the
first one, after the initial block of 1's, that is not zero.
(All $p_i$ are odd, so the first block is forced.) Suppose
there are $s$ 1's and $r-s$ 0's in that block.  The probability
of obtaining such a block is
$$
\frac{{r \choose s}}{2^r - 1}
$$
since there are $r \choose s$ binary tuples with Hamming weight $s$,
and $2^r - 1$ blocks that force a stop.

Given this information, what is the probability that
$p_1 p_2 \ldots p_r$ is 2-Korselt?  It is 1 when $s=r$ (regardless
of parity), and it is 0 when $s$ is odd with $s<r$.

The remaining case ($s$ even, $s<r$) can be analyzed as follows.
We continue the process, revealing only enough bits to determine
$\nu_{s+1},\ldots,\nu_r$.  Whether or not the 2-Korselt property
holds is determined solely by the unseen bits.  Order the $p_i$'s
so that $p_1$ and $p_2$ have the minimum exponent, and
now reveal all of $p_3,\ldots,p_r$.  Then,
$$
p_1 p_2 \ldots p_r \equiv 1 \pmod {2^\nu}
$$
iff
$$
u_1 (1 + 2^\mu u_2) + u_2 \equiv
\frac{(p_3 \cdots p_r)^{-1} - 1} {2^\mu} \pmod {2^{\nu-\mu}} .
$$
The right hand side is integral, and even because the $p_i$ with
exponent $\mu$ come in pairs.  The coefficient of $u_1$ is odd.
Therefore, for each possible $u_2$ (odd), there is exactly one
way to choose $u_1$ (odd) mod $2^{\nu - \mu}$ so as to make the
above congruence true.  Since $\nu-\mu-1$ bits of $u_1$ are
now forced, we have (for these $s$)
$$
\Pr[\hbox{ $p_1 \ldots p_r$ is 2-Korselt } | \mu, \nu ]
= \frac{1}{2^{\nu-\mu-1}}.
$$

To summarize,
$$
\Pr[\hbox{ $p_1 \cdots p_r$ is 2-Korselt } | \nu_1,\ldots,\nu_r ]
= \begin{cases}
  1, &\mbox{ if $s=r$;} \\
  0, & \mbox{if $1 \le s < r$ and $s$ is odd;} \\
  \frac{1}{2^{\nu-\mu-1}}, & \mbox{if $2 \le s < r$ and $s$ is even.} \\
\end{cases}
$$
(Note that $s,\mu,\nu$ are all functions of $\nu_1,\ldots,\nu_r$.)
When $s<r$ is even, the random variable $\nu - \mu$, necessarily
1 or greater, has the same distribution $Z$ in the lemma, but with
$n = r-s$.  Therefore,
$$
\Pr \big[ \hbox{ $p_1 \cdots p_r$ is 2-Korselt } |\ s \big]
= \begin{cases}
  1, & \mbox{if $s=r$;} \\
  0, & \mbox{if $1 \le s < r$ and $s$ is odd;} \\
  W(r-s), & \mbox{if $2 \le s < r$ and $s$ is even.} \\
  \end{cases} 
$$
The theorem now follows from the lemma and the conditional
probability formula $\Pr[K] = E[\Pr[K|s]]$.

Exact values of the probabilities, which are rational, can be
readily computed from the theorem.  Here, we list a few of them,
and their decimal values.

$$
\begin{matrix}
1 &2 &3 &4 &5 &6 &7 &8 &9 &10 \\
\\
1 &\frac {1} {3} &\frac {3} {7} &\frac {9} {35} &\frac {167} {651}
&\frac {43} {217} &\frac {725} {3937} &\frac {95339} {602361}
&\frac {24834279} {171003595} &\frac {49160655} {376207909} \\
\\
1.000 &0.333 &0.429 &0.257 &0.257 &0.198 &0.184 &0.158 &0.145 &0.131 \\
\end{matrix}
$$

We claimed that when $r \ge 3$, a product of $r$ random odd 2-adic
integers is 2-Korselt with probability $\Theta(r^{-1})$.  This
follows from the two theorems below.

\bigskip

\begin{thm}As $r \rightarrow \infty$,
$$
\Pr[\hbox{ $p_1 \ldots p_r$ is 2-Korselt } ] = \Omega(r^{-1}).
$$
\end{thm}

\begin{proof}  Let $Z$ be as in the lemma.  It can be shown that
$$
\frac{H_n}{\log 2}
\le E[Z] \le
\frac{H_n}{\log 2} + 1 .
$$
Also, from comparison with the integral,
$$
H_n \le \log n + 1.
$$
Therefore, by Jensen's inequality,
$$
E[2^{-Z}] \ge \frac{1}{4 n} .
$$
Therefore, the probability in question is at least
$$
\frac{1}{2^r}
+
\frac{1}{2^r}
\sum_{2 \le s < r \atop s\ {\rm even}} {r \choose s} \frac{1}{4(r-s)}.
$$
The first term is exponentially small and can be neglected.
Since ${r \choose s} = {r \choose s'}$ when $s' = r-s$, we can
rewrite the second term as
$$
\frac{1}{4}
\sum_{s' \in A} {r \choose s'} \frac{2^{-r}} {s'} ,
$$
where $A = \{s' : 1 \le s' \le r-2 \hbox{ and } s'\equiv r (2)\}$.
The sum above equals $E[(s')^{-1} | A] \Pr[A]$, with
$s' \sim \hbox{binomial}(r,1/2)$.
By Jensen's inequality
$$
E[(s')^{-1} | A]
\ge \frac{1}{E[s'|A]}
\ge \frac{\Pr[A]} {E[s']}
= \frac {2 \Pr[A]} {r}.
$$
The claimed result now follows, since $\Pr[A] = 1/2 + o(1)$.
\end{proof}

\bigskip

\begin{thm}As $r \rightarrow \infty$,
$$
\Pr[\hbox{ $p_1 \ldots p_r$ is 2-Korselt } ] = O(r^{-1}).
$$
\end{thm}

\begin{proof}
Consider $f(t) = e^{-t} (1 - e^{-t})^n$.
This vanishes at 0 and $+\infty$, and is
nonnegative when $t \ge 0$.  Moreover, since
$$
f'(t) = e^{-t}(1 - e^{-t})^{n-1} \left[ (n+1)e^{-t} - 1 \right],
$$
$f$ is unimodal (increases, then decreases) and is maximized when
$e^t = (n+1)$.  Its maximum value is
$$
\frac{1}{n+1} \left( 1 - \frac{1}{n+1} \right)^n
\le \frac{1}{n+1}.
$$
Let $\alpha = \log 2$.  Then,
$$
W(n) = \sum_{k \ge 1} 2^{-k} (1 - 2^{-k})^n
     = \sum_{k \ge 1} f(\alpha k)
$$
$$
     \le \int_0^\infty f(\alpha t) dt + 2 \max\{ f(\alpha t) \}
     = \frac{1}{\alpha(n+1)} + \frac{2}{n+1}
     \le \frac{4}{n+1}.
$$
Using symmetry as before, and including omitted terms (they
are all positive), we get
$$
\Pr[\hbox{ $p_1 \ldots p_r$ is 2-Korselt } ]
\le
\frac{1}{2^r + 1} +
\frac{4}{2^r + 1} \sum_{s' \ge 1} {r \choose s'} \frac{1}{s'+1}.
$$
Only the second term matters, and it equals
$$
\frac {2^{r+3}} {(2^r + 1)(r+1)}
\sum_{s' \ge 2} {r+1 \choose s'} 2^{-(r+1)}
= O(r^{-1}),
$$
since binomial probabilities sum to 1.
\end{proof}

\end{document}